\documentclass[12pt]{article}

\usepackage[utf8]{inputenc}
\usepackage{amsmath}
\usepackage{amsfonts}
\usepackage{amssymb}
\usepackage{mathtools}

\allowdisplaybreaks

\usepackage{subfigure}

\usepackage{comment}

\usepackage{amsthm}

\theoremstyle{lema}

\theoremstyle{proposition}

\theoremstyle{theorem}
\newtheorem{theorem}{Theorem}[section]

\theoremstyle{theorem}
\newtheorem{remark}{Remark}[section]

\theoremstyle{corollary}

\theoremstyle{theorem}
\newtheorem{definition}{Definition}[section]

\theoremstyle{claim}
\newtheorem{claim}{Claim}
\usepackage{chngcntr}
\counterwithin*{claim}{theorem}

\def\r{\mathbb R}
\def\n{\mathbb N}

\begin{document}
	
\title{A countable fractal interpolation scheme involving Rakotch contractions}


\author{Cristina Maria PACURAR\\\small{Faculty of Mathematics and Computer Science,}\\ \small{Transilvania University of Bra\c sov, Bulevardul Eroilor 29, Bra\c sov}\\
\small{email:{cristina.pacurar@unitbv.ro}  } }

\date{}	

\maketitle

	\begin{abstract}
		The main result of this paper states that for a given countable system of data $\Delta$, there exists a countable iterated function system consisting of Rakotch contractions, such that its attractor is the graph of a fractal interpolation function corresponding to $\Delta$. In this way, on the one hand, we generalize a result due to N. Secelean (see \textit{The fractal interpolation for countable systems of data}, Univ. Beograd. Publ. Elektrotehn. Fak. Ser. Mat., \textbf{14} (2003), 11–19) by considering countable systems consisting of Rakotch contractions rather than Banach contractions. On the other hand, we generalize a result due to S. Ri (see \textit{A new idea to construct the fractal interpolation function}, Indag. Math., \textbf{29} (2018), 962-971) by considering countable (rather than finite) systems consisting of Rakotch contractions. Some exemplifications are provided.
	\end{abstract}



\section{Introduction}

Fractal interpolation is a special method for constructing a continuous function which passes through all of the points of a given system of points. For a given set of data $\{(x_i,y_i) \in I \times \r, i = \{0,\dots ,N-1\}\}$, where $I=[x_0,x_N]$ is a closed real interval and $x_i < x_{i+1}$ for all $i = \{0,\dots,N\}$, the fractal interpolation function (FIF) is a continuous function $f:I\to \r$ which interpolates the given data such that its graph is the attractor of an iterated function system, a notion due to Hutchinson (see \cite{Hutchinson}). Fractal interpolation functions were introduced by Barnsley (see \cite{Barnsley}, \cite{Barnsley-book}) and have been intensively studied ever since. 

The main difference between fractal interpolation and other types of interpolation techniques is that the interpolation function obtained is not necessarily differentiable at any point, thus, being closer to natural world phenomena and providing a more powerful tool in fitting real-world data. A comprehensive survey on FIFs is that of Navascu\'es et al. (see \cite{survey}).

In the development of the theory of FIFs, there have been many generalizations of Barnsley's result. Among these directions of research, we mention the hidden variable FIFs, introduced by Barnsley for systems of data which are not self-referential (see \cite{Barnsley}, \cite{Barnsley-hidden}, \cite{Chand}, \cite{Dalla}) and the extension to higher dimensional cases of FIFs (see \cite{Massopust}, \cite{Dalla-2}, \cite{Dalla-3}, \cite{Xie}, \cite{Zhao}, \cite{Ri-2}, \cite{Ruan}, \cite{Verma}).

Another direction of interest regarding FIFs is related to the fixed point result which guarantees the existence of the FIF. While most of the extensions rely on the Banach fixed point theorem (following Barnsley's results) in order to prove the existence of a FIF, there have been recent results which use different fixed point results. In this respect, Ri has used Rakotch contractions to obtain new results (see \cite{Ri}), Kim et al. resorted to Geraghty contractions (see \cite{Kim}) and Ri and Drakopoulos extended the results to surfaces (see \cite{Ri-3}).

A different direction related to FIFs is to extend the finite set of points which are interpolated to a countable set. Thus, countable fractal interpolation has been introduced by Secelean (see \cite{Secelean-1}) based on countable iterated function systems (see \cite{Fernau}, \cite{Secelean-cifs}, \cite{Secelean-CIFS}). In \cite{Secelean-1}, Secelean proved the existence of the FIF for a countable iterated function system for a set of data $\Delta=\{(x_n,y_n) \in I \times \r, n \geq 0\}$ where $(x_n)_{n \geq 0}$ is a strictly increasing bounded sequence and $(y_n)_{n\geq 0}$ is a convergent sequence. These results were extended by Gowrisankar and Uthayakumar (see \cite{Gow}) for systems of data where $(x_n)_{n\geq 0}$ is a monotone bounded sequence and $(y_n)_{n \geq 0}$ is a bounded sequence. This direction has been further developed in several papers (see \cite{Secelean-rep}, \cite{Secelean-3}, \cite{Vis}).

By combining these two lines of research initiated by Secelean and Ri, in this paper we present a new fractal interpolation scheme for countable systems of data and countable iterated function systems composed of Rakotch contractions. Thus, the present paper extends the results from \cite{Secelean-1} and \cite{Ri}. Although the techniques that we used in our proofs are similar to those from \cite{Secelean-1} and \cite{Ri}, the countable iterated function systems for Rakotch contractions requires highly more effort and subtleties.

\section{Preliminaries}

\subsection{Notations and terminology}
	
Let $(X,d)$ be a compact metric space.

We denote by $\mathcal{P}_{cp}(X)$ the set of all non-empty compact subsets of $X$. 

We consider the Hausdorff metric $h : \mathcal{P}_{cp}(X) \times \mathcal{P}_{cp}(X) \to [0, \infty)$, defined as \begin{equation*}
	h(A,B) = \max \{\underset{x \in A}{\sup}\; \underset{y \in B}{\inf} d(x,y), \underset{x \in B}{\sup}\; \underset{y \in A}{\inf} d(x,y)\}
\end{equation*}
for $A, \, B \in \mathcal{P}_{cp} (X)$.

For $A \subset X$, by $diam(A)$ we denote the diameter of $A$.

\begin{definition}
	Given a metric space $(X,d)$, an operator $f:X\to X$ is called a Picard operator if $f$ has a unique fixed point $x_* \in X$ and $$\lim\limits_{n\to \infty}f^{[n]}(x) = x_*,$$ for every $x \in X$, where by $f^{[n]}$ we mean the composition of $f$ with itself n-times.
\end{definition}

\subsection{Rakotch and Matkowski contractions}

\begin{definition}[see Matkowski \cite{Matkowski}, Rakotch \cite{Rakotch}, Jachymski \cite{Jachymski}, Rhoades \cite{Rhoades}]
	
	~\begin{itemize}
		\item[i)] Let $\varphi : [0,\infty) \to [0,\infty)$ and $(X,d)$ a metric space. A map $f: X \to X$ is called a $\varphi$-contraction if $$d(f(x),f(y)) \leq \varphi(d(x,y)),$$ for all $x,\, y \in X$.
		
		\item[ii)] Given a metric space $(X,d)$, a map $f: X \to X$ is called Matkowski contraction if it is a $\varphi$-contraction, where $\varphi:[0,\infty) \to [0,\infty)$ is non-decreasing and $\lim\limits_{n\to \infty} \varphi^{[n]}(t) = 0$ for all $t > 0$.
		
		\item[iii)] Given a metric space $(X,d)$, a map $f: X \to X$ is called Rakotch contraction  if it is a $\varphi$-contraction, where $\varphi:[0,\infty) \to [0,\infty)$ is such that the function $\alpha : (0, \infty) \to (0,\infty)$, given by $\alpha(t) = \frac{\varphi(t)}{t}$ for every $t > 0$ is non-increasing and $\alpha(t) < 1$ for every $t \in (0,\infty)$.
	\end{itemize}
\end{definition}

\begin{remark} [see Remark 2.2 from \cite{Ri}, \cite{Jachymski-2} and \cite{Chand}]
	Given a metric space $(X,d)$, a map $f: X \to X$ is a Rakotch contraction if and only if it is a $\varphi$-contraction for some non-decreasing $\varphi : [0,\infty) \to [0,\infty)$ such that the function $\alpha : (0, \infty) \to (0,\infty)$, given by $\alpha(t) = \frac{\varphi(t)}{t}$ for every $t \in (0,\infty)$ is non-increasing and $\alpha(t) < 1$ for every $t \in (0,\infty)$.
\end{remark}

\begin{remark}
	\begin{itemize}
	\item[i)] Each Banach contraction is a Rakotch contraction  (for a function $\varphi$ given by $\varphi(t)=\alpha t$ for every $t >0$, where $\varphi \in [0,1)$).
	\item[ii)] Each Rakotch contraction is a Matkowski contraction.
	\end{itemize}
	\label{remark1}
\end{remark}

\begin{theorem} [see Matkowski \cite{Matkowski}]
	Given a complete metric space $(X,d)$, if $f : X \to X$ is a Matkowski contraction, then $f$ has a unique fixed point $x_* \in X$ and $\lim\limits_{n\to \infty}  f^{[n]}(x) = x_*$ for each $x \in X$.
	\label{MatkowskiTh}
\end{theorem}

\begin{remark}
	The above theorem says that each Matkowski contraction on a complete metric space is a Picard operator.
\end{remark}

\subsection{Countable iterated function systems}

\begin{definition}
	Let $(X,d)$ be a compact metric space and $f_n : X \to X$ be  continuous functions for every $n \in \n$. The pair $\mathcal{S} = ((X,d), (f_n)_{n \geq 1})$ is called a countable iterated function system (for short, CIFS).
\end{definition}

The fractal operator associated to the CIFS $\mathcal{S} = ((X,d), (f_n)_{n \geq 1})$ is the function $F_{\mathcal{S}} : \mathcal{P}_{cp}(X) \to \mathcal{P}_{cp}(X)$, defined as $$\displaystyle F_{\mathcal{S}}(K) = \overline{\underset{n \geq 1}{\bigcup}f_n(K)}$$ for every $K \in \mathcal{P}_{cp}(X)$. 

If the fractal operator $F_{\mathcal{S}}$ is Picard, then we say that the CIFS $\mathcal{S}$ has attractor and the fixed point of $F_{\mathcal{S}}$ is called the attractor of the CIFS $\mathcal{S}$.

\begin{theorem}[see Secelean \cite{JMAA-Secelean} Theorem 3.7]
	If the constitutive functions $f_n$ of the CIFS $\mathcal{S} = ((X,d), (f_n)_{n \geq 1})$ are Matkowski contractions, for every $n \geq 1$, then the fractal operator $F_{\mathcal{S}}$ is a Matkowski contraction. In particular, the CIFS has attractor.
	\label{F-Matkowski}
\end{theorem}

\subsection{Countable systems of data and interpolation functions}

Given a compact metric space $(Y,d)$, let us consider the countable system of points 
\begin{equation}
	\Delta = \{(x_n,y_n) \in \r \times Y, n \geq 0\}.
	\label{countableSys}
\end{equation}

If the sequence $(x_n)_{n \geq 0}$ is strictly increasing and bounded, and the sequence $(y_n)_{n \geq 0}$ is convergent, then the system of points defined in relation (\ref{countableSys}) is called a countable system of data. 

We set the notations $a = x_0$, $b = \lim\limits_{n\to \infty} x_n$, $m = y_0$ and $M = \lim\limits_{n\to \infty} y_n$. 

\begin{definition}
	In the above mentioned framework, an interpolation function corresponding to the countable system of data $\Delta$ is a continuous function $f:[a,b] \to Y$, such that $$f(x_n) = y_n,$$ for each $n \geq 0$.
\end{definition}

Note that
\begin{equation*}
\begin{aligned}
f(b) & = f(\lim\limits_{n \to \infty} x_n) \\
& \overset{f \, \text{continuous}}{=} \lim\limits_{n \to \infty} f(x_n) \\
& = \lim\limits_{n \to \infty}y_n\\
& = M.
\end{aligned}
\end{equation*}

\subsection{The family $(f_n)_n$ associated to $\Delta$}

Let $\Delta = \{(x_n,y_n) \in \r \times Y, n \geq 0\}$ be a countable system of data.

For each $n\geq 1$, let $l_n : [a,b] \to [x_{n-1}, x_n]$ be a homeomorphism for which there exists $L_n \in [0,1)$ such that
\begin{itemize}
\item[i)] \begin{equation*}
	|l_n(x) - l_n(x')| \leq L_n |x - x'|
\end{equation*}
for every $x, \, x' \in [a,b]$;
\item[ii)]
\begin{equation*}
	l_n(a) = x_{n-1} \quad \text{and} \quad l_n(b) = x_n;
	\label{ln-1}
\end{equation*}
\item[iii)] $$\underset{n \geq 1}{\sup} L_n <1.$$
\end{itemize}

For each $n\geq 1$, let $W_n : [a,b] \times Y \to Y$ be a continuous function such that
\begin{itemize}
\item[j)] \begin{equation*}
	W_n(a,m) = y_{n-1}\quad \text{ and } \quad W_n(b,M) = y_n;
	\label{defW_n}
\end{equation*}
\item[jj)] $\lim\limits_{n \to \infty} diam(Im\, W_n) = 0$.
\end{itemize}

For $n \geq 1$, we define $f_n : [a,b] \times Y \to [a,b] \times Y$ as
\begin{equation*}
	f_n(x,y) =(l_n(x), W_n(x,y)),
\end{equation*}
for every $x \in [a,b]$ and $y \in Y$.

\subsection{The operator $\textbf{T}$}

Let us consider $\mathcal{C}([a,b]) = \{ f:[a,b] \to Y \lvert\,f(a) = m \text{ and } f(b) = M, f\text{ - continuous}\}$ endowed with the uniform metric $d_{\mathcal{C}([a,b])}$.

\begin{remark}
	 The space $(\mathcal{C}([a,b]), d_{\mathcal{C}([a,b])})$ is a complete metric space.
\end{remark}

Let $\Delta = \{(x_n,y_n) \in \r \times Y, n \geq 0\}$ be a countable system of data.

For $f \in \mathcal{C}([a,b])$, we consider the function $\textbf{T}_f: [a,b] \to Y$ given as follows:
\begin{equation*}
	\textbf{T}_f(x) = \begin{cases}
	\begin{aligned}
	&W_n(l_n^{-1}(x), f(l_n^{-1}(x))),& \quad &\text{ if } x \in [x_{n-1},x_n]&\\
	&M, & \quad  &\text{ if } x = b.&
	\end{aligned}
	\end{cases}
\end{equation*}

\begin{claim}
	$\textbf{T}_f$ is well defined.
	\label{claim1}
\end{claim}

Indeed, since $x_n \in [x_{n-1},x_n]$, by definition of $\textbf{T}_f$, on the one hand we have
	\begin{equation*}
	\begin{aligned}
	\textbf{T}_f(x_n) &= W_n(l_n^{-1}(x_n), f(l_n^{-1}(x_n)))\\
	&= W_n(b, f(b))\\
	&= W_{n}(b, M)\\
	&= y_n
	\end{aligned}
	\end{equation*}
	and on the other hand, since $x_n \in [x_{n},x_{n+1}]$, we have 
	\begin{equation*}
	\begin{aligned}
	\textbf{T}_f(x_n) &= W_{n+1}(l_{n+1}^{-1}(x_{n}), f(l_{n+1}^{-1}(x_{n})))\\
	&= W_{n+1}(a, f(a))\\
	&= W_{n+1}(a, m)\\
	&= y_n,
	\end{aligned}
	\end{equation*}
for all $n \geq 1$. 

\begin{claim}
	$\textbf{T}_f \in \mathcal{C}([a,b])$.
	\label{claim2}
\end{claim}

Indeed, on the one hand we have
\begin{equation*}
	\begin{aligned}
		\textbf{T}_f(a) & = W_1(l_1^{-1}(a), f(l_1^{-1}(a))) \\
			& = W_1(a,f(a)) \\
			& = W_1(a,m) \\
			& = y_0\\
			& = m
	\end{aligned}
\end{equation*}
and by definition $$\textbf{T}_f(b) = M.$$

On the other hand, since $W_n$ are continuous, it is clear that $\textbf{T}_f$ is continuous on $(x_{n-1},x_n)$ for all $n \geq 1$. We need to prove that $\textbf{T}_f$ is right continuous at $a$, continuous at $x_n$ for all $n \geq 1$ and left continuous at $b$.

For $n \geq 1$, we have
\begin{equation*}
	\begin{aligned}
		\lim\limits_{x \searrow x_n} \textbf{T}_f(x) & = \lim\limits_{x \searrow x_n}W_{n+1}(l_{n+1}^{-1}(x), f(l_{n+1}^{-1}(x))) \\
		& = W_{n+1}(a,f(a)) \\
		& = W_{n+1}(a,m)\\
		& = y_n\\
		& = \textbf{T}_f(x_n)
	\end{aligned}
\end{equation*}
and
\begin{equation*}
\begin{aligned}
\lim\limits_{x \nearrow x_n} \textbf{T}_f(x) & = \lim\limits_{x \nearrow x_n}W_{n}(l_{n}^{-1}(x), f(l_{n}^{-1}(x))) \\
& = W_{n}(b,f(b)) \\
& = W_n(b,M) \\
& = y_n\\
& = \textbf{T}_f(x_n),
\end{aligned}
\end{equation*}
which proves that $Tf$ is continuous on $(a,b)$.

Since 
\begin{equation*}
\begin{aligned}
	\lim\limits_{x \searrow a} \textbf{T}_f(x) &= \lim\limits_{x \searrow a} W_1(l_1^{-1}(x), f(l_1^{-1}(x)))\\
	& = W_{1}(a,f(a)) \\
	& = W_{1}(a,m)\\
	& = y_0\\
	& = \textbf{T}_f(a),
\end{aligned}	
\end{equation*}
we infer that $\textbf{T}_f$ is right continuous at $a$. 

Now we prove that $\textbf{T}_f$ is left continuous at $b$.

Let $\varepsilon > 0$ be fixed, but arbitrary chosen. 

As $\lim\limits_{n \to \infty} y_n = M$ and $\lim\limits_{n \to \infty} diam(Im\, W_n) = 0$, there exists $n_{\varepsilon} \in \n$ such that 
\begin{equation}
	d(M- y_n) < \frac{\varepsilon}{2}
	\label{My_n}
\end{equation}
and 
\begin{equation}
	diam(W_n) < \frac{\varepsilon}{2}
	\label{My_n2}
\end{equation}
for every $n \geq 1$, $n \geq n_{\varepsilon}$.

For $x \in (x_{n_{\varepsilon}}, b)$, as $(x_n)_n$ is a strictly increasing sequence and $\lim\limits_{n \to \infty} x_n = b$, there exists $n_x \in \n$, $n_x \geq n_{\varepsilon}$ such that $x \in [x_{n_x}, x_{n_x+1}]$, so we have
\begin{equation*}
\begin{aligned}
 	d(\textbf{T}_f(x), \textbf{T}_f(b)) &\leq d(\textbf{T}_f(x), \textbf{T}_f(x_{n_x})) + d(y_{n_x}, M) \\
 	& = d(W_{n_x+1}(l_{n_x+1}^{-1}(x), f(l_{n_x+1}^{-1}(x))), W_{n_x+1}(l_{n_x+1}^{-1}(x_{n_x}), f(l_{n_x+1}^{-1}(x_{n_x})))) \\
 	& \qquad \qquad \qquad \qquad + d(M, y_{n_x})\\
 	&\leq diam(Im W_{n_x+1}) + d(M,y_{n_x}) \\
 	& \overset{(\ref{My_n}) \, \& \, (\ref{My_n2})}{\leq} \frac{\varepsilon}{2} + \frac{\varepsilon}{2}\\
 	& =\varepsilon.
\end{aligned}
\end{equation*}
Hence, \begin{equation*}
\begin{aligned}
\lim\limits_{x \nearrow b} \textbf{T}_f(x)  = \textbf{T}_f(b), \text{ i.e. }
\end{aligned}
\end{equation*}
$\textbf{T}_f$ is left continuous at $b$, which concludes the proof that $\textbf{T}_f$ is continuous on $[a,b]$.

\vspace{20pt}

Thus, from Claim \ref{claim1} and Claim \ref{claim2}, the operator $T: \mathcal{C}([a,b])  \to \mathcal{C}([a,b])$, defined as $$T(f) = \textbf{T}_f$$ for every $f \in \mathcal{C}([a,b])$ is well defined.

\section{Main results}

\begin{theorem}
	Let $\Delta = \{(x_n,y_n) \in \r \times Y, n \geq 0\}$ be a countable system of data. If the functions $W_n$ are Matkowski contractions with respect to the second argument, i.e. there exists a non-decreasing function $\varphi : [0,\infty) \to [0,\infty)$ such that $\lim\limits_{n \to\infty} \varphi^n(t) = 0$ for all $t>0$ and 
	\begin{equation}
		d(W_n((x,y)),W_n((x,y'))) \leq \varphi(d(y,y'))
		\label{MatkWn}
	\end{equation}
	for all $x \in [a,b]$ and $y,y'\in Y$, then $T$ is a Matkowski contraction. 
	\label{T-contraction}
\end{theorem}

\begin{proof}
	Let $g,h \in \mathcal{C}([a,b])$. 
	
	It is obvious that
	\begin{equation}
	0 = d(M, M) = d(Tg(b),Th(b)) \leq \varphi(d_{\mathcal{C}([a,b])}(g,h)).
	\label{Tinb}
	\end{equation}
	 
	Let $x\in[a,b)$ and $n \geq 1$ such that $x \in [x_{n-1},x_n]$. 
	
	Then, we have
	\begin{equation}
	\begin{aligned}
		d(Tg(x),Th(x)) & =  d(W_n(l_n^{-1}(x), g(l_n^{-1}(x))), W_n(l_n^{-1}(x), h(l_n^{-1}(x)))) \\ &	\overset{(\ref{MatkWn})}{\leq} \varphi(d(g(l_n^{-1}(x)),h(l_n^{-1}(x)))) \\ & \leq \varphi( \underset{u \in [a,b]}{\sup}d(g(u),h(u))) \\ &= \varphi(d_{\mathcal{C}([a,b])}(g,h)).
	\end{aligned}
	\label{Tall}
	\end{equation}
	
	Via (\ref{Tinb}) and (\ref{Tall}), we get 
	\begin{equation*}
		d_{\mathcal{C}([a,b])}(Tg,Th) = \underset{x \in [a,b]}{\sup} d(Tg(x),Th(x)) \leq \varphi(d_{\mathcal{C}([a,b])}(g,h)),
	\end{equation*}
	which concludes our proof.
\end{proof}

\begin{theorem}
	Let $\Delta = \{(x_n,y_n) \in \r \times Y, n \geq 0\}$ be a countable system of data such that $W_n$ are Lipschitz with respect to the first variable and Rakotch contractions in the second variable, i.e. there exists $L > 0$, and a non-decreasing function $\varphi : [0,\infty) \to [0,\infty)$, satisfying $\frac{\varphi(t)}{t}< 1$ for all $t > 0$ and $t \to \frac{\varphi(t)}{t}$ is non-increasing, such that 
	\begin{equation*}
		d(W_n((x,y)),W_n((x',y'))) \leq L|x-x'| + \varphi(d(y,y'))
	\end{equation*}
	for all $(x,y),\,(x',y') \in [a,b] \times Y$, $n\geq 1$.
	
	Then, $f_n $ are Rakotch contractions with respect to the metric $d_{\theta}$ described by 
	\begin{equation*}
	d_{\theta} ((x,y),(x',y')) :=  |x-x'|+\theta d(y,y')
	\end{equation*}
	for all $(x,y),\,(x',y') \in [a,b] \times Y$, where $\theta = \frac{1- \underset{n \geq 1}{\sup}L_n}{2(L+1)} \in (0,1)$.
	\label{f_nRakotch}
\end{theorem}

\begin{proof}
	For all $(x,y),\, (x',y') \in [a,b]\times Y$, $(x,y) \neq (x',y')$ and $n \geq 1$, we have
		\begin{align*}
		 d_{\theta} (f_n(x,y),f_n(x',y')) & = d_{\theta} ((l_n(x),W_n(x,y)),(l_n(x'),W_n(x',y'))) \\
		 & = |l_n(x) - l_n(x')| + \theta d(W_n(x,y), W_n(x',y'))\\
		 & \leq L_n |x-x'| + \theta (L|x-x'| + \varphi(d(y,y'))) \\
		 & = (L_n + \theta L) |x-x'| \\&  \qquad \qquad + \theta \dfrac{\varphi(d(y,y'))}{|x-x'| + d(y,y')} (|x-x'| + d(y,y'))
		 \\&\overset{\varphi \text{ non-decreasing}}{\leq} (L_n + \theta L) |x-x'| + \\&  \qquad \qquad  \theta \dfrac{\varphi(|x-x'| + d(y,y'))}{|x-x'| + d(y,y')} (|x-x'| + d(y,y'))\\
		 & =  \left[L_n + \theta \left(L + \dfrac{\varphi(|x-x'|+d(y,y'))}{|x-x'| + d(y,y')}\right)\right] |x-x'| \\&  \qquad \qquad + \theta \dfrac{\varphi(|x-x'|+ d(y,y'))}{|x-x'| + d(y,y')} d(y,y') \\
		 & \overset{\overset{t \to \frac{\varphi(t)}{t}}{ \text{ non-increasing}}}{\leq}  \left[L_n + \theta \left(L + \dfrac{\varphi(|x-x'|+d(y,y'))}{|x-x'| + d(y,y')}\right)\right] |x-x'| \\&  \qquad \qquad + \theta \dfrac{\varphi(|x-x'|+ \theta d(y,y'))}{|x-x'| + \theta d(y,y')} d(y,y') \\
		 & \overset{\frac{\varphi(t)}{t} < 1,\, (\forall)\, t >0}{\leq} [\underset{n \geq 1}{\sup}L_n + \theta \left(L + 1\right) ]|x-x'| \\& \qquad \qquad + \theta \dfrac{\varphi(|x-x'|+ \theta d(y,y'))}{|x-x'| + \theta d(y,y')} d(y,y').
		\end{align*}

	Thus, we get 
	\begin{equation*}
	\begin{gathered}
		 d_{\theta} (f_n(x,y),f_n(x',y')) \\
		\leq \max\left\{ \underset{n \geq 1}{\sup} L_n + \theta \left(L + 1\right) , \dfrac{\varphi(d_{\theta}((x,y),(x',y')))}{d_{\theta}((x,y),(x',y'))}\right\} d_{\theta}((x,y),(x',y'))
	\end{gathered}
	\end{equation*}
	for all $(x,y),\, (x',y') \in [a,b]\times Y$, $(x,y) \neq (x',y')$.
	
	Let us consider the map $\alpha : (0,\infty) \to (0,1)$ defined as
	\begin{equation*}
		\alpha(t) = \max\left\{ \underset{n \geq 1}{\sup} L_n + \theta \left(L + 1\right), \dfrac{\varphi(t)}{t}\right\}
	\end{equation*}
	for all $t > 0$, and $n \geq 1$.
	
	Since $\frac{\varphi(t)}{t} < 1$ for every $t >0$ and
	\begin{equation*}
	\begin{aligned}
		\underset{n \geq 1}{\sup} L_n + \theta \left(L + 1\right) & =  \underset{n \geq 1}{\sup} L_n +\frac{1- \underset{n \geq 1}{\sup}L_n}{2(L+1)} \left(L + 1\right) \\
		& = \underset{n \geq 1}{\sup} L_n + \frac{1-\underset{n \geq 1}{\sup}L_n}{2} \\
		& <1,
	\end{aligned}
	\end{equation*} 
	it is clear that $\alpha(t) \in (0,1)$ for every $t > 0$.
	
	Since $t\to \frac{\varphi(t)}{t}$ is non-increasing, we infer that $\alpha$ is non-increasing.
	
	Thus, considering $\psi : [0, \infty) \to [0, \infty)$, given by $\psi(t) =  t \alpha(t)$ for all $t\geq 0$, we get 
	\begin{equation*}
	\begin{aligned}
	d_{\theta} (f_n(x,y),f_n(x',y')) 
	&\leq \alpha(d_{\theta}((x,y),(x',y'))) d_{\theta}((x,y),(x',y'))\\
	& = \psi(d_{\theta}((x,y),(x',y'))),
	\end{aligned}
	\end{equation*}
	for every $(x,y),\, (x',y') \in [a,b]\times Y$.
	
	Since $\dfrac{\psi(t)}{t} = \alpha(t) < 1$ for every $t>0$ and $\alpha$ is non-increasing, we conclude that $f_n$ are Rakotch contractions with respect to $d_{\theta}$.
\end{proof}

Let $(x_n,y_n) \subseteq [a,b]\times Y$. Since $Y$ is compact, there exists a subsequence $(y_{n_k})_{n}$ of $(y_n)_{n}$ and $y \in Y$ such that $\lim\limits_{k \to \infty}y_{n_k} = y$. Since $[a,b]$ is compact, there exists a subsequence $(x_{n_{k_p}})_p$ of $(x_{n_k})_k$ and $x \in [a,b]$ such that $\lim\limits_{p \to \infty} x_{n_{k_p}} = x$. We have $\lim\limits_{p \to \infty} d_{\theta}(x_{n_{k_p}}, y_{n_{k_p}}) = (x,y) \in [a,b]\times Y$. Thus, $([a,b]\times Y, d_{\theta})$ is compact, and from Theorem \ref{F-Matkowski} and Remark \ref{remark1}, we get the following:

\begin{remark}
	Let $\Delta = \{(x_n,y_n) \in \r \times Y, n \geq 0\}$ be a countable system of data. If the functions $f_n$ are Rakotch contractions with respect to the metric $d_{\theta}$ (in particular, if the conditions stated in Theorem 3.2 are satisfied), then the CIFS $\mathcal{S} = (([a,b] \times Y,d_{\theta}), (f_n)_{n\geq 1})$ has attractor, so there exists a unique $A_{\mathcal{S}} \in \mathcal{P}_{cp}([a,b]\times Y)$ such that 
	\begin{equation*}
	F_{\mathcal{S}}(A_{\mathcal{S}})= A_{\mathcal{S}}.
	\end{equation*}
	\label{Remark2}
\end{remark}

\begin{theorem}
	Let $\Delta = \{(x_n,y_n) \in \r \times Y, n \geq 0\}$ be a countable system of data such that $W_n$ satisfy the hypothesis from Theorem \ref{f_nRakotch}. Then there exists an interpolation function $f_*$ corresponding to $\Delta$ such that its graph is the attractor of the countable iterated function system $\mathcal{S} = (([a,b] \times Y,d_{\theta}), (f_n)_{n\geq 1})$. 
	\label{T3}
\end{theorem}
\begin{proof}
	Since $W_n$ are Rakotch contractions with respect to the second argument, from Theorem \ref{T-contraction} and Remark \ref{remark1}, we get that $T$ is a Matkowski contraction. Thus, it has a unique fixed point $f_* \in \mathcal{C}([a,b])$. Hence,
	$$Tf_*(x) = f_*(x)$$ 
	for all $x \in [a,b]$. 
	
	For $n \geq 1$ and $x\in [x_{n-1}, x_n]$, we get
	\begin{equation*}
		Tf_*(x) = W_n(l_n^{-1}(x), f_*(l_n^{-1}(x))) = f_*(x),
	\end{equation*}
	and $$Tf(b) = M = f_*(b).$$
	
	As
	\begin{equation*}
	\begin{aligned}
		f_*(x_n) &= Tf_*(x_n) \\
			& = W_n(l_n^{-1}(x_n), f_*(l_n^{-1}(x_n))) \\
			& = W_n(b,f_*(b))\\
			& = W_n(b,M) \\
			& = y_n 
	\end{aligned}
	\end{equation*}
	for every $n \geq 1$, we conclude that $f_*$ is an interpolation function corresponding to $\Delta$.
	
	Let $G$ be the graph of $f_*$. 
	
	\begin{claim}
	\begin{equation*}
		\displaystyle\overline{\underset{n \geq 1}{\bigcup}f_n(G)} \subseteq G.
	\end{equation*}
	\end{claim}

	\textit{Justification of Claim 1}
	
	We have 
	\begin{equation}
	\begin{aligned}
	f_*(l_n(x)) & = Tf_*(l_n(x)) \\
	& = W_n(l_n^{-1}(l_n(x)), f_*(l_n^{-1}(l_n(x))))\\
	& = W_n(x, f_*(x)),
	\end{aligned}
	\label{f*}
	\end{equation}
	for every $x \in [a,b]$.
	
	Thus, we get
	\begin{equation}
	\begin{aligned}
		f_n(x, f_*(x)) & = (l_n(x), W_n(x, f_*(x))) \\
		& \overset{(\ref{f*})}{=} (l_n(x), f_*(l_n(x))) \in G 
	\end{aligned}
	\label{f**}
	\end{equation}
	for every $n \geq 1$ and every $x\in [a,b]$, so $$\underset{n \geq 1}{\bigcup}f_n(G) \subseteq G.$$
	
	Since $G$ is closed, we have 
	\begin{equation*}
	\displaystyle\overline{\underset{n \geq 1}{\bigcup}f_n(G)} \subseteq G.
	\label{subset1}
	\end{equation*}
	
	\begin{claim}
		\begin{equation*}
			G \subseteq	\displaystyle\overline{\underset{n \geq 1}{\bigcup}f_n(G)}
		\end{equation*}
	\end{claim}

	\textit{Justification of Claim 2}
	
	If $x \in[a,b)$ then there exists $n \geq 1$ such that $x \in [x_{n-1}, x_n]$, so we have
	\begin{equation}
	\begin{aligned}
		(x, f_*(x)) & = (x, f_*(l_n(l_n^{-1}(x))))\\
		& \overset{(\ref{f*})}{=} (l_n(l_n^{-1}(x)), W_n(l_n^{-1}(x), f_*(l_n^{-1}(x)))) \\
		& \overset{(\ref{f**})}{=} f_n(l_n^{-1}(x), f_*(l_n^{-1}(x))) \in f_n(G) \subseteq \displaystyle\overline{\underset{n \geq 1}{\bigcup}f_n(G)}
		\label{G2}
	\end{aligned}
	\end{equation}
	
	Moreover, we have $(b, f_*(b)) = \lim\limits_{n \to \infty} (x_n, f_*(x_n))$ and since $(x_n, f_*(x_n)) \overset{(\ref{G2})}{\in} f_n(G)$, for every $n \geq 1$ we infer that $(b, f_*(b)) \in \displaystyle\overline{\underset{n \geq 1}{\bigcup}f_n(G)}$, so
	\begin{equation*}
		 G \subseteq \displaystyle\overline{\underset{n \geq 1}{\bigcup}f_n(G)}.
		 \label{subset2}
	\end{equation*}
	
	Thus, from the above two claims, we get the equality 
	\begin{equation*}
	\displaystyle\overline{\underset{n \geq 1}{\bigcup}f_n(G)} = G,
	\end{equation*}
	i.e.
	\begin{equation*}
	F_{\mathcal{S}}(G) = G.
	\end{equation*}
	
	As $d_1$ and $d_{\theta}$ are equivalent, where $d_1((x,y), (x',y')) = |x- x'| + d((y,y'))$ for all $(x,y), (x',y') \in [a,b] \times Y$, we have  $G\in \mathcal{P}_{cp}([a,b]\times Y)$ and via Remark \ref{Remark2} we obtain $G= A_{\mathcal{S}}$, i.e.
	\begin{equation*}
		A_{\mathcal{S}} = \{(x,f_*(x))\lvert x \in [a,b]\}.
		\label{uniqueAtractor}
	\end{equation*}
\end{proof}

\begin{theorem}
	Under the framework of Theorem \ref{T3}, we have 
	\begin{equation*}
		\lim\limits_{n \to \infty}T^{[n]}(f_0) = f_*
	\end{equation*}
	for every $f_0 \in \mathcal{C}([a,b])$.
\end{theorem}

\begin{proof}
	
	\begin{claim}
		$$ G_{T(f_0)} = F_{\mathcal{S}}(G_{f_0}),$$
		for every $f_0 \in \mathcal{C}([a,b])$.
	\end{claim}
	\textit{Justification of Claim}
	
	For the very beginning, let us note that 
	\begin{equation}
		\displaystyle\underset{n \geq 1}{\bigcup}f_n(G_{f_0}) \subseteq G_{T(f_0)},
		\label{Claim1}
	\end{equation}
	for every $f_0 \in \mathcal{C}([a,b])$.
	
	Indeed, if $y \in \underset{n \geq 1}{\bigcup}f_n(G_{f_0})$, then there exists $n_y \geq 1$, such that $y \in f_{n_y}(G_{f_0})$, so there exists $x \in [a,b]$ having the property that $y  =f_{n_y}(x, f_0(x))$.
	
	Hence 
	\begin{equation*}
		\begin{aligned}
			y & = (l_{n_y}(x), W_{n_y}(x, f_0(x)))\\
			& = (l_{n_y}(x), W_{n_y}(l_{n_y}^{-1}(l_{n_y}(x)), f_0(l_{n_y}^{-1}(l_{n_y}(x)))))\\
			& \overset{l_{n_y}(x) \in [x_{n_y-1},x_{n_y}]}{=} (l_{n_y}(x), T(f_0)(l_{n_y}(x))) \in G_{T(f_0)},
		\end{aligned}
	\end{equation*}
	and the proof of (\ref{Claim1}) is completed.
	
	We have
	\begin{equation}
		F_{\mathcal{S}}(G_{f_0}) \subseteq G_{T(f_0)}
		\label{Claim1-1}
	\end{equation}
	for every $f_0 \in \mathcal{C}([a,b])$.
	
	Indeed, if $ y \in 	F_{\mathcal{S}}(G_{f_0}) = \overline{\underset{n \geq 1}{\bigcup}f_n(G_{f_0})}$, then there exists $(z_k)_{k \geq 1} \subseteq \underset{n \geq 1}{\bigcup}f_n(G_{f_0})$ such that $ y =\lim\limits_{k \to \infty} z_{k}$.
	
	Hence, according to (\ref{Claim1}), there exists $u_k \in [a,b]$ with the property that $y = \lim\limits_{k \to \infty} (u_k, T(f_0)(u_k)) \in \overline{G_{T(f_0)}} \overset{T(f_0) \, \text{continuous}}{=} G_{T(f_0)}$.
	
	We have 
	\begin{equation}
	G_{T_{f_0}} \subseteq F_{\mathcal{S}}(G_{f_0})
	\label{Claim1-2}
	\end{equation}
	for every $f_0 \in \mathcal{C}([a,b])$.
	
	Indeed, let us consider $y \in G_{T(f_0)}$. Then, there exists $x \in [a,b]$, such that $y = (x, T(f_0)(x))$. First, let us note that if $x \in [a,b)$, there exists $n_x \geq 1$ having the property that $x\in[x_{n_x-1}, x_{n_x}]$, so one can find 
	$u_x \in [a,b]$ such that $ x = l_{n_x}(u_x)$.
	
	Thus, 
	\begin{equation*}
	\begin{aligned}
		y & = (l_{n_x}(u_x), T(f_0)(l_{n_x}(u_x)))\\
		 & \overset{l_{n_x}(x) \in [x_{n_x-1},x_{n_x}]}{=} (l_{n_x}(u_x), W_{n_x}(l_{n_x}^{-1}(l_{n_x}(u_x)),f_0(l_{n_x}^{-1}(l_{n_x}(u_x)))))\\
		 & = (l_{n_x}(u_x), W_{n_x}(u_x,f_0(u_x)))\\
		 & = f_{n_x}(u_x,f_0(u_x)) \in f_{n_x}(G_{f_0}) \subseteq 
		 \underset{n \geq 1}{\bigcup}f_n(G_{f_0}) \subseteq \overline{\underset{n \geq 1}{\bigcup}f_n(G_{f_0})} = F_{\mathcal{S}}(G_{f_0}).
	\end{aligned}
	\end{equation*}
	
	Consequently, 
	\begin{equation}
		(x,T(f_0)(x)) \in F_{\mathcal{S}}(G_{f_0})
		\label{Claim1-3}
	\end{equation}
	for every $x \in [a,b)$ and every $f_0\in \mathcal{C}([a,b])$.
	
	In addition,
	\begin{equation}
		(b,T(f_0)(b)) = \lim\limits_{x \nearrow b} (x,T(f_0)(x)) \in \overline{F_{\mathcal{S}}(G_{f_0})} \overset{G_{f_0} \text{ compact} }{=} F_{\mathcal{S}}(G_{f_0})
		\label{Claim1-4}
	\end{equation}
	for every $f_0 \in \mathcal{C}([a,b])$.
	
	Relations (\ref{Claim1-3}) and (\ref{Claim1-4}) ensure (\ref{Claim1-2}).
	
	Taking into account (\ref{Claim1-1}) and (\ref{Claim1-2}), the justification of the Claim is completed.
	
	Finally, the Claim implies - via the mathematical induction method - that 
	\begin{equation}
		F_{\mathcal{S}}^{[n]}(G_{f_0}) = G_{T^{[n]}(f_0)}
		\label{Claim1-5}
	\end{equation}
	for every $n \geq 1$ and every $f_0 \in \mathcal{C}([a,b])$.
	
	As $G_{f_0} \in \mathcal{P}_{cp}([a,b]\times Y)$, we have 
	\begin{equation*}
	\lim\limits_{n \to \infty} F_{\mathcal{S}}^{[n]}(G_{f_0}) \overset{\text{ Remark \ref{Remark2}}}{=} A_{\mathcal{S}}
	\end{equation*}
	so, via (\ref{Claim1-5}) and Theorem \ref{T3}, we get 
	\begin{equation*}
		\lim\limits_{n \to\infty} G_{T^{[n]}(f_0)} = G_{f_*}.
	\end{equation*}
		
\end{proof}

\section{Particular cases}

We can choose
\begin{equation}
l_n(x) = \dfrac{x_n-x_{n-1}}{b-a}x + \dfrac{bx_{n-1}-ax_n}{b-a}
\label{l_n}
\end{equation}
for every $x \in [a,b]$.

It is immediate that 
\begin{equation*}
	\begin{aligned}
		|l_n(x) - l_n(x')| &\leq \left|\dfrac{x_n-x_{n-1}}{b-a}x + \dfrac{bx_{n-1}-ax_n}{b-a} - \dfrac{x_n-x_{n-1}}{b-a}x' - \dfrac{bx_{n-1}-ax_n}{b-a}\right| \\
		&\leq \dfrac{x_n-x_{n-1}}{b-a}\left|x-x'\right|
	\end{aligned}
\end{equation*}
where $\dfrac{x_n-x_{n-1}}{b-a} \in [0,1)$ and $\underset{n \geq 1}{\sup}\dfrac{x_n-x_{n-1}}{b-a} < 1$. Also, we have 
\begin{equation*}
\begin{aligned}
	&l_n(a) &=& \dfrac{x_n-x_{n-1}}{b-a}a + \dfrac{bx_{n-1}-ax_n}{b-a} \\
	& &=&x_{n-1}\\
	&l_n(b) &=& \dfrac{x_n-x_{n-1}}{b-a}b + \dfrac{bx_{n-1}-ax_n}{b-a} \\
	& &=&x_{n}
\end{aligned}
\end{equation*}
for every $n \geq 1$.

Note that if $Y$ is a compact real interval, we can choose $W_n$ in the following two ways:

\textbf{A.}
\begin{equation}
W_n(x,y) = c_nx+d_ny+g_n, 
\label{exW_n}
\end{equation}
where $$c_n = \dfrac{y_n-y_{n-1}}{b-a} - d_n\dfrac{M-m}{b-a},$$  $$g_n = \dfrac{by_{n-1}-ay_n}{b-a} - d_n\dfrac{bm-aM}{b-a}$$ and $d_n \in [0,1)$ such that $\lim\limits_{n \to \infty} d_n = 0$.

Indeed, on the one hand we have
\begin{equation*}
\begin{aligned}
W_n(a,m) &= \left[\dfrac{y_n-y_{n-1}}{b-a} - d_n\dfrac{M-m}{b-a}\right]a + d_n m + \dfrac{by_{n-1}-ay_n}{b-a} - d_n\dfrac{bm-aM}{b-a} \\
& = \dfrac{ay_n-ay_{n-1}+by_{n-1}-ay_n}{b-a} - d_n\dfrac{aM-am+bm-aM-m(b-a)}{b-a}\\
& = y_{n-1}
\end{aligned}
\end{equation*}
and 
\begin{equation*}
\begin{aligned}
W_n(b,M) &= \left[\dfrac{y_n-y_{n-1}}{b-a} - d_n\dfrac{M-m}{b-a}\right]b + d_n M + \dfrac{by_{n-1}-ay_n}{b-a} - d_n\dfrac{bm-aM}{b-a} \\
& = \dfrac{by_n-by_{n-1}+by_{n-1}-ay_n}{b-a} - d_n\dfrac{bM-bm+bm-aM-M(b-a)}{b-a}\\
& = y_{n}
\end{aligned}
\end{equation*}
for every $n \geq 1$. 

On the other hand, we have 

\begin{equation*}
\begin{aligned}
0 \leq diam(Im\, W_n) & = \underset{(x,y),(x',y')\in[a,b]\times Y}{\sup}|W_{n}(x, y) - W_{n}(x',y')| \\
& = \underset{(x,y),(x',y')\in[a,b]\times Y}{\sup} |(c_nx + d_ny + g_n) - (c_nx' + d_ny' + g_n)| \\
& \leq \underset{(x,y),(x',y')\in[a,b]\times Y}{\sup} |c_n||x-x'| + |d_n||y-y'|\\
&\leq (b-a)c_n + diam(Y)d_n
\end{aligned}
\end{equation*}
for all $n\geq 1$.

As $\lim\limits_{n \to \infty} c_n = \lim\limits_{n\to \infty} d_n = 0$, we get $\lim\limits_{n \to \infty} diam(Im\, W_n) = 0$.

Note that \begin{equation*}
\begin{aligned}
|W_n(x,y)-W_n(x',y')| & = |c_n x + d_n y + g_n - (c_n x' + d_n y' + g_n)|\\ &\leq \left| c_n\right| |x-x'| +d_n\left|y-y'\right|\\&\leq L|x-x'|+c|y-y'|
\end{aligned}
\end{equation*}
for all $(x,y), (x',y') \in [a,b]\times Y$, where $L = \underset{n \geq 1}{\sup}|c_n| \in \r$ and $c= \underset{n\geq 1}{\sup} |d_n| \in [0,1)$.

Hence, $W_n$ are Banach contractions on the second variable, so they are Rakotch contractions on the second  variable for the comparison function $\varphi$ given by $\varphi(t) = c\cdot t$ for every $t \geq 0$.

\textbf{B.}

\begin{equation}
W_n(x,y) = c_nx+\dfrac{y}{1+ny}+g_n, 
\label{exW_nB}
\end{equation}
where $$c_n = \dfrac{y_n-y_{n-1}}{b-a} - \dfrac{1}{b-a}\left(\dfrac{M}{1+nM} - \dfrac{m}{1+nm}\right)$$ 
and 
$$g_n = y_{n-1} - a\dfrac{y_n - y_{n-1}}{b-a} + \dfrac{a}{b-a} \dfrac{M}{1+nM} - \dfrac{b}{b-a} \dfrac{m}{1+nm}.$$

Indeed, on the one hand we have
\begin{equation*}
\begin{aligned}
W_n(a,m) &= \left[\dfrac{y_n-y_{n-1}}{b-a} - \dfrac{1}{b-a}\left(\dfrac{M}{1+nM} - \dfrac{m}{1+nm}\right)\right]a \\&+ \dfrac{m}{1+ nm} + y_{n-1} - a\dfrac{y_n - y_{n-1}}{b-a} + \dfrac{a}{b-a} \dfrac{M}{1+nM} - \dfrac{b}{b-a} \dfrac{m}{1+nm} \\
& = y_{n-1}
\end{aligned}
\end{equation*}
and 
\begin{equation*}
\begin{aligned}
W_n(b,M) &= \left[\dfrac{y_n-y_{n-1}}{b-a} - \dfrac{1}{b-a}\left(\dfrac{M}{1+nM} - \dfrac{m}{1+nm}\right)\right]b \\&+ \dfrac{M}{1+ nM} + y_{n-1} - a\dfrac{y_n - y_{n-1}}{b-a} + \dfrac{a}{b-a} \dfrac{M}{1+nM} - \dfrac{b}{b-a} \dfrac{m}{1+nm}\\
& = y_{n},
\end{aligned}
\end{equation*}
for every $n \geq 1$. 

On the other hand, we have 

\begin{equation*}
\begin{aligned}
0 \leq diam(Im\, W_n) & = \underset{(x,y),(x',y')\in[a,b]\times Y}{\sup}|W_{n}(x, y) - W_{n}(x',y')| \\
& = \underset{(x,y),(x',y')\in[a,b]\times Y}{\sup} \left|\left(c_nx + \frac{y}{1+ny} + g_n\right) - \left(c_nx' + \dfrac{y'}{1+ny'} + g_n\right)\right| \\
& \leq \underset{(x,y),(x',y')\in[a,b]\times Y}{\sup} 
\left(|c_n||x-x'| + |y-y'| \dfrac{1}{(1+ny)(1+ny')}\right)\\
&\leq (b-a)c_n + diam(Y)\dfrac{1}{(1+n\inf Y)^2}
\end{aligned}
\end{equation*}
for every $n\geq 1$.

As $\lim\limits_{n \to \infty} c_n = \lim\limits_{n\to \infty} \dfrac{1}{(1+ n \inf y)^2} = 0$, we get $\lim\limits_{n \to \infty} diam(Im\, W_n) = 0$.

Note that if $Y \subseteq [0,\infty)$, we have 
\begin{equation*}
\begin{aligned}
|W_n(x,y)-W_n(x',y')| & = \left|\left(c_nx + \frac{y}{1+ny} + g_n\right) - \left(c_nx' + \dfrac{y'}{1+ny'} + g_n\right)\right|\\ &\leq \left| c_n\right|\cdot |x-x'| + \dfrac{\left|y-y'\right|}{(1+ny)(1+ny')}\\
&\leq L |x-x'| + \dfrac{\left|y-y'\right|}{1+|y-y'|}\\
&= L |x-x'| + \varphi(|y-y'|)
\end{aligned}
\end{equation*}
for all $(x,y), (x',y') \in [a,b]\times Y$, where $L = \underset{n \geq 1}{\sup}|c_n| \in \r$ and the comparison function $\varphi$ is given by $\varphi(t) = \dfrac{t}{1+t}$. 

Hence, $W_n$ are Rakotch contractions on the second variable, but they are not Banach contractions on the second variable. 

In the particular case where $Y$ is a compact real interval we can choose $f_n$ in the following two ways:

\textbf{A.} 
\begin{equation*}
\begin{aligned}
f_n(x,y)= &\left( \dfrac{x_n-x_{n-1}}{b-a}x + \dfrac{bx_{n-1}-ax_n}{b-a}, \right. \\ &\;\;\; \left. \left(\dfrac{y_n-y_{n-1}}{b-a} - d_n\dfrac{M-m}{b-a}\right) x + d_n y + \dfrac{by_{n-1}-ay_n}{b-a} - d_n\dfrac{bm-aM}{b-a} \right).
\end{aligned}
\end{equation*}

\textbf{B.}
\begin{equation*}
\begin{aligned}
f_n(x,y)= &\left( \dfrac{x_n-x_{n-1}}{b-a}x + \dfrac{bx_{n-1}-ax_n}{b-a}, \right. \\ &\;\;\; \left. \left(\dfrac{y_n-y_{n-1}}{b-a} - \dfrac{1}{b-a}\left(\dfrac{M}{1+nM} - \dfrac{m}{1+nm}\right)\right)x + \dfrac{y}{1+ny} \right. \\& \left.+ y_{n-1} - a\dfrac{y_n - y_{n-1}}{b-a} + \dfrac{a}{b-a} \dfrac{M}{1+nM} - \dfrac{b}{b-a} \dfrac{m}{1+nm} \right).
\end{aligned}
\end{equation*}

The above considerations show that our result is a genuine generalization of Secelean's and Ri's results.

\end{document}